\documentclass[a4paper,11pt,twoside,leqno]{article}
\usepackage{amsmath,amssymb,amsfonts,amsthm,graphicx,fancyhdr,xcolor}
\usepackage[utf8]{inputenc}
\usepackage[T1]{fontenc}
\usepackage{rotating}
\usepackage{thmtools}
\usepackage{mathtools}
\usepackage{hyperref}
\usepackage{authblk}

\newtheorem{teo}{Theorem}[section]
\newtheorem{lem}[teo]{Lemma}
\newtheorem{prop}[teo]{Proposition}
\newtheorem{cor}[teo]{Corollary}
\newtheorem{dfn}[teo]{Definition}
\newtheorem{ques}[teo]{Question}

\newtheorem{claim}[teo]{Claim}

\newtheorem*{remk}{Remark}

\DeclareMathOperator{\Cut}{Cut}
\DeclareMathOperator{\Sep}{Sep}

\declaretheoremstyle[
  spaceabove=\topsep, spacebelow=\topsep,
  headfont=\bf,  
  notefont=\mdseries, notebraces={(}{)},
  bodyfont=\rmfamily, 
  postheadspace=1em,
  qed=$\Diamond$
]{drem}

\declaretheorem[style=drem, name=Example, numberlike=teo]{exa}

\newcommand{\ie}[0]{\emph{i.e.} }

\newcommand{\ce}{\vcentcolon =}

\DeclareSymbolFont{mbbold}{U}{bbold}{m}{n}
\DeclareSymbolFontAlphabet{\mathbbold}{mbbold}

\newcommand{\inj}[0]{\hookrightarrow}
\newcommand{\surj}[0]{\twoheadrightarrow}

\newcommand{\zz}[0]{\ensuremath{\mathbb{Z}}}
\newcommand{\Z}[0]{\ensuremath{\mathbb{Z}}}

\newcommand{\A}[0]{\mathbb{A}}
\newcommand{\C}[0]{\mathbb{C}}
\newcommand{\F}[0]{\mathbb{F}}
\newcommand{\Q}[0]{\mathbb{Q}}
\newcommand{\R}[0]{\mathbb{R}}
\DeclareMathOperator{\cha}{char}

\newcommand{\bsl}\backslash

\title{Regular maps from the lamplighter to metabelian groups}
\author{Antoine Gournay and Corentin Le Coz\thanks{Supported by the Israel Science Foundation (grant no. 2919/19), and part of the project EOS-40007542 ``High-dimensional expanders and Kac–Moody–Steinberg groups''}}
\affil{Institut für Geometrie, TU Dresden, Germany, antoine.gournay@gmail.com\\
Department of Mathematics: Algebra and Geometry (WE01), Ghent Univesity, corentinlecoz@outlook.com}

\begin{document}

\maketitle
\abstract{We prove that the lamplighter group admits an injective Lipschitz map to any finitely generated metabelian group which is not virtually nilpotent.
This implies that finitely generated metabelian groups satisfy the ``analytically thin/analytically thick'' dichotomy recently introduced by Hume, Mackay and Tessera.
}

\section{Introduction}\label{s-intro}

The study of regular maps, namely Lipschitz maps satisfying that the preimage of every point is bounded (see \S\ref{sec: reg maps} for more details), dates back at least to David and Semmes \cite{davidsemmes}.
Benjamini, Schramm and Tim\'ar introduced in \cite{BST} a corresponding invariant called \emph{separation profile}, in the spirit of the celebrated theorem of Lipton and Tarjan \cite{liptontarjan1} and further work from Miller, Teng, Thurston, and Vavasis \cite{MTT1}.

\begin{dfn}[Benjamini, Schramm, Tim\'ar \cite{BST}]
	Let $G$ be a graph with bounded degree.
	We define the separation profile of $G$ to be the function $\mathbb N \to \mathbb N$ given by:
	\[\Sep_G(v) = \sup_{F \subset VG, |F| \le v} \Cut F,\]
	where $\Cut F$ is the minimal size of subsets $C$ of $F$ satisfying that the connected components of $F\setminus C$ contain at most $|F|/2$ vertices.
\end{dfn}

In this definition, and throughout this paper, we identify subsets of vertices to the corresponding induced subgraph.
As usual, we will endow separation profiles with the partial order given by $f \preceq g$ if and only if the exists $C>0$ such that $f(v) \preceq C g(Cv)$ for every $v \ge C$, and denote by $\asymp$ the corresponding equivalence relation.
The main interest of separation profiles in geometric group theory is their monotonicity under regular maps:

\begin{prop}[Benjamini, Schramm, Tim\'ar \cite{BST}]
Let $G, H$ be graphs with bounded degree such that there exists a regular map $H \to G$.
Then, $\Sep_H \preceq \Sep_G$.
\end{prop}

This property enables separation profiles to give obstruction to the existence of regular maps between graphs of bounded degree or finitely generated groups.
Regular maps arise naturally in geometric group theory: quasi-isometric embeddings, subgroup inclusions of finitely generated groups and more generally coarse embeddings are examples of such maps.

Very few such invariants are known, and it is still a very active field of research.
Separation profiles were generalized by Hume, Mackay and Tessera \cite{HMT} into a spectrum of invariants $\Lambda^p$, called Poincar\'e profiles, where the $\Lambda^1$ profile is equivalent to the separation profile.
We refer to the thesis of the second author \cite{lecoz2020thesis} for more details on this subject.

The authors of the present paper proved in \cite{CG} that every finitely generated solvable group $G$ satisfying $\Sep_G(v) \preceq v^{1-\epsilon}$, for some positive $\epsilon$, must be virtually nilpotent.
A remarkable application of this result was obtained by Tessera \cite{TesHyp}, who proved that every amenable group having a regular map into a cocompact lattice in $\mathbb H^n \times \mathbb R^d$ must be virtually nilpotent.

Some explicit groups are known to have a separation profile equivalent to $\frac v{\log v}$, a natural function dominating every function of the form $v^{1-\epsilon}$:
direct products of non abelian free groups \cite{BST},
Baumslag-Solitar groups $\mathrm{BS}(m,n)$ when $|m| \neq |n|$ \cite[Theorem 1.16]{HMTLie},
the split oscillator group $\mathrm{Osc} = \mathrm{Heis}_3 \rtimes_{(1,-1,0)} \mathbb R$ \cite[Theorem 4.6]{HMTLie},
the lamplighter group $\mathbb Z_2 \wr \mathbb Z$, \cite[Theorem 4.3]{HMTLie} (see also Proposition \ref{propseplamp} of the present paper).
It is then natural to ask in which groups one can embed the examples above; in a previous paper \cite[Question 7.4]{CG}, the authors asked whether $\mathbb Z_2 \wr \mathbb Z$ coarsely embeds into any exponential growth solvable group.

Such constructive results were obtained by Hume, Mackay and Tessera in \cite{HMTLie}: Theorem 1.11, giving that $\mathbb Z_2 \wr \mathbb Z$ or $\mathrm{Osc}$ quasi-isometrically embeds into every finitely generated polycyclic group having exponential growth, or Theorem 3.1 stating that $\mathbb Z_2 \wr \mathbb Z$ quasi-isometrically embeds into $\mathrm{BS}(m,n)$ if $|m| \neq |n|$, and in $\mathrm{SOL}_a$ for every $a>0$.

Here, we will focus on the lamplighter group $\mathbb Z_2 \wr \mathbb Z$ and show how it can be embedded into metabelian groups.
The main result of this paper is:
\begin{teo}
\label{teomet}
	Let $G$ be a finitely generated metabelian group having exponential growth.
	Then, there exists an injective regular map $\mathbb Z_2 \wr \mathbb Z \to G$.
\end{teo}
We recall that an injective regular map is nothing but an injective Lipschitz map.
Using Milnor-Wolf theorem and the computation of separation profiles of virtually nilpotent groups \cite{HMT,CG}, we obtain the following dichotomy for metabelian groups.

\begin{cor}
\label{cor: dichotomy}
	Let $G$ be a finitely generated metabelian group.
	Then,
	\begin{itemize}
		\item either $G$ is virtually nilpotent and satisfies $\Sep_G(v) \asymp v^{1-1/d}$, with $d$ being the degree of growth of $G$,
		\item or there exists a regular map $\mathbb Z_2 \wr \mathbb Z \to G$ and $\Sep_G(v) \succeq \frac v{\log v}$.
	\end{itemize}
\end{cor}

Following Definition 1.3. from \cite{HMTLie}, the first and the second case corresponds to the situation where $G$ is respectively \emph{analytically thin} and \emph{analytically thick}.
Our result implies that every finitely generated metabelian group is either analytically thin or analytically thick.
This partially answers Questions 8.3 and 8.5 from \cite{HMTLie}, which are stated in the more general context of solvable groups.
This statement can be compared with the dichotomies obtained by Hume Mackay and Tessera in \cite[Theorem 1.5]{HMTLie} for connected unimodular Lie groups and \cite[Corollary 1.6]{HMTLie} for polycyclic groups.

A very natural and interesting special case of \cite[Question
8.5]{HMTLie} is whether $\mathbb Z_2 \wr \mathbb Z$ has a regular maps into $\mathrm{Osc}$.
A positive answer would immediately imply that $\mathbb Z_2 \wr \mathbb Z$ has a regular map into every finitely generated polycyclic group, using  \cite[Theorem 1.11]{HMTLie}.
\\

\textit{About proof techniques:} Let us recall that an affine group is a subgroup of $k \rtimes k^\times$, where $k$ a field.
Theorem \ref{teomet} is obtained in two steps.
First (\S\ref{sec: lamp and affine}), we construct regular maps from the lamplighter group to metabelian groups that are quotients of an affine group satisfying some conditions.
Second (\S\ref{sec: proof themeta}), we use previous results of Tits and Groves implying that we can realize every finitely generated metabelian group as a group satisfying the conditions of the previous step.\\

\paragraph{Comments and questions}
As explained above, it is not known whether there exists a regular map from the lamplighter group to every finitely generated solvable group which is not virtually nilpotent.

It is also natural to ask whether we can improve Theorem \ref{teomet} with coarse or quasi-isometric embeddings instead of regular maps.
\begin{ques}
	Let $G$ be a finitely generated metabelian group having exponential growth.
	Does there exists a coarse/quasi-isometric embedding $\mathbb Z_2 \wr \mathbb Z \to G$?
\end{ques}

Corollary \ref{cor: dichotomy} does not tell what the separation profile of finitely generated metabelian groups having exponential growth is; it only gives a lower bound.
On the other hand, groups whose separation profile is not dominated by $\frac v{\log v}$ are quite rare (see  \cite[Corollary 5.6]{CG}).
We can therefore ask the following question:
\begin{ques}
Do we have $\Sep_G(v) \preceq \frac v{\log v}$ for every finitely generated metabelian group?
\end{ques}
Another natural question is:
\begin{ques}
Let $G$ be a finitely generated group of exponential growth. Is there a quasi-isometric embedding of a regular (non-amenable) tree in $G$?
\end{ques}
See de Cornulier \& Tessera \cite{CorTes} for related work in the case of soluble groups and Benjamini \& Schramm \cite{BenSch} for the case of non-amenable groups.\\

\textit{Organization of the paper:}
In \S\ref{sec: reg maps}, we recall the basic definitions, give few examples and prove Theorem \ref{teomet}.
In \S\ref{sec: separation of lamplighter}, we give an elementary proof of the computation of the separation profile of the lamplighter group.\\

\textit{Acknowledgements:}
We are grateful to David Hume for interesting discussions and comments on an earlier version of the paper.

\section{Regular maps}
\label{sec: reg maps}

\begin{dfn}\cite{BST}
Let $A$ and $B$ be simplicial graphs with bounded degree.
A map $\phi: A \to B$ is said to be regular if the two following conditions hold:\\
{\bfseries (R1)} There is a $K>0$ such that $d_B\big( \phi(x),\phi(y) \big) \leq K d_A(x,y)$ for every $x,y \in A$,\\
{\bfseries (R2)} There is a $C \geq 1$ such that $|\phi^{-1}(x)| \leq C$ for every $x \in A$.
\end{dfn}

Observe that a composition of regular maps is regular (one can make the constants explicit).
The difference between Cayley graph and groups will be omitted (since the choice of Cayley graph does not play a role).

\subsection{Some simple examples}

As explained in the introduction, quasi-isometric embeddings, subgroup inclusions of finitely generated groups and more generally coarse embeddings are examples of regular maps.
Let us give some examples of regular maps from the lamplighter group $L = \zz_2 \wr \zz$ to other groups.
The first example is the lamplighter group $\zz \wr \zz$.
\begin{exa}
\label{example: zz}
Let $W =\zz \wr \zz$ generated by the usual generating sets $s$ is the switch (\ie as a function $\zz \to \zz$ it is the Dirac mass at 0) and $w$ is the walk (a generator of $\zz$). 
Consider the same generating set for $L$ (except that now $s$ is a function $\zz \to \zz_2$).
The map $L \to W$ which is the inclusion (seeing the set $\zz_2 = \{0,1\}$ as a subset of $\zz$) is a regular map, with $K=C=1$.
\end{exa}

This second example is Baumslag-Solitar groups and introduces the main ideas that will come into play in the proof of Theorem \ref{teomet}.
\begin{exa}
Here, we consider soluble Baumslag-Solitar groups of exponential growth (or the [exponential growth] metabelianisation of the non-soluble ones).
More precisely let $p$ and $q$ be coprime integers with $pq \neq \pm 1$, and let $M_{p,q} = \zz[\tfrac{1}{pq}] \rtimes \zz$ (where $\zz$ acts by multiplication by $\tfrac{p}{q}$).
$M_{p,q}$ is generated by $a = 1 \in \zz[\tfrac{1}{pq}]$ and $b = 1 \in \zz$.
The generator~$a$ acts by~$(x,i)a = (x + \frac{p^i}{q^i}, i)$, and~$b$ acts by~$(x,i)b = (x,i+1)$.

Note that $M_{p,q}$ is a quotient of $W = \zz \wr \zz$.
Indeed, since~$\zz[\tfrac{1}{pq}]$ is abelian and normal in~$M_{p,q}$,~$a$ and its conjugates commute.
The composition of the regular map from $L = \zz_2 \wr \zz$ to $W$ with the quotient map $W \surj M_{p,q}$ is given by $(f,i) \mapsto \left( \sum_{j \in \Z} f(j) \frac{p^j}{q^j} , i \right)$.
Basically, this map can be obtained by writing any $\ell \in L$ as a reduced word and then transposing this word in $M_{p,q}$, with $s \mapsto a$ and $w \mapsto b$.
This is a regular map, again with $K=C=1$.
\end{exa}

%
%
%
No regular map $M_{p,q} \to L$ exist because $M_{p,q}$ has asymptotic dimension $2$ and $L$ has asymptotic dimension $1$.


%
%
%
\begin{exa}
A typical example of a finitely generated amenable but non solvable group is $F \wr \zz$, where $F$ is a finite simple group.
This group is quasi-isometric to $\zz_F \wr \zz$, which gives a natural quasi-isometric embedding of $\zz_2 \wr \zz$, with a similar approach as in Example \ref{example: zz}.
\end{exa} 
\begin{exa}
Another example of an amenable non-solvable group is $\mathrm{Sym}_{\mathrm{fin}}(\zz) \rtimes \zz$ where by $\mathrm{Sym}_{\mathrm{fin}}(\zz)$ we denote the group of finitely supported permutations on $\zz$, and $\zz$ acts by shifting (the infinite cyclic permutation). In that case sending $s \mapsto$ some finite permutation and $w \mapsto$ a large enough power of the shift (so as to send the permutation to one with a disjoint support) provides an injective regular map $L \to \mathrm{Sym}_{\mathrm{fin}}(\zz) \rtimes \zz$. 
\end{exa}
%
Note that the quotient map $\pi: W \surj M_{p,q}$ is not a regular map because it has infinite preimages.
Many other interesting groups, such as those constructed by Brieussel and Zheng \cite{brieusselzheng} contain a lamplighter group as a subgroup.

\subsection{Metabelian groups}

This section is split into two parts, \S\ref{sec: lamp and affine} where we prove how to embed the lamplighter group into metabelian sections of some affine groups, see Lemma \ref{lemnous}, and \S\ref{sec: proof themeta} were we prove Theorem \ref{teomet} in the case of metabelian groups.

\subsubsection{Lamplighter and affine groups}
\label{sec: lamp and affine}
%
%
%
%

Let $k$ be a field. Then $\A(k)$ will denote the group of affine transformations of $k$. 
It can be seen as the set of bijective maps $x \mapsto a x + b$ (where $a \in k^\times$ and $b \in k$), where the group operation is composition.
This gives a natural action on $k$. 
Alternatively, it consists in the $2 \times 2$ matrices which can be written as $\big(\begin{smallmatrix}a & b \\ 0 & 1 \end{smallmatrix}\big)$ (again $a \in k^\times$ and $b \in k$).
Lastly, it is can also be seen as the semi-direct product $k \rtimes k^\times$ (where $k^\times$ acts by multiplication on $k$).
Consequently, there is a natural projection $\pi: \A(k) \surj k^\times$.

Recall that a local field is  $\R$, $\C$, a finite extension of $\Q_p$ or a field of Laurent series over a finite field, endowed with the appropriate norm.

\begin{lem}\label{lemnous}
Let $k$ be a local field with norm $|\cdot|_k$.
Let $\Gamma$ be a non-abelian finitely generated subgroup of $\A(k)$ such  that there exists $\gamma \in \Gamma$ such that $|\pi(\gamma)|_k \neq 1$.
Let $G$ be a finitely generated metabelian group such that there exists a surjective homomorphism $\psi\colon G \surj \Gamma$.
Then there is a regular map from $L$ to $G$.
\end{lem}

\begin{proof}
Since $\Gamma$ is not abelian, it contains two elements $\gamma_1$ and $\gamma_2 \in \Gamma_0$ which do not commute. 
Thus $c = [\gamma_1,\gamma_2]$ is non trivial and is of the form $(b,1)$ for some $b \in k - \{0\}$.
In other words, it belongs to the abelian [additive] subgroup $k$ inside $\A(k)$.
This additive subgroup act by translation on the line $k$.

Let now $\gamma \in \Gamma$ satisfying $|\pi(\gamma)|_k \neq 1$.
It acts on the line $k$ by a homothety (or dilation) with a unique fixed point.
For convenience, one can conjugate $\Gamma$ with a translation of $\A(k)$ so that the fixed point of $\gamma$ is $0$.
Indeed, $G$ maps onto each conjugate of $\Gamma$ in $\A(k)$, since they are all isomorphic to $\Gamma$.
Concretely, $\gamma$ is now of the form $(0,a)$, for some $a \in k^\times$ satisfying $|a|_k \neq 1$, and $c$ is unchanged.
Up to raising $\gamma$ to a suitable power, we can assume that we have $|a|_k \ge 2$.

Let now $d, \delta \in G$ be such that $\psi(d) = c$ and $\psi(\delta) = \gamma$.

An element of $L$ can be written as $(\oplus_{j \in \zz} \ell_j, i)$ where only finitely many $\ell_j$ are non-zero (in $\Z_2$) and $i \in \zz$ is the lamplighter position.
Given the sequence $\ell_j$ each belonging to $\Z_2$, let $\ell_j'$ be the sequence of 0 and 1 ($\in \zz$) obtained by abusing notations.
Then we define:
\begin{align*}
  \phi \colon L & \longrightarrow G  \\
  \big(\oplus_{j \in \zz} \ell_j, i \big) & \longmapsto \big( \prod_{j \in \zz} \delta^j d^{\ell_j'} \delta^{-j} \big) \delta^i
\end{align*}
We should mention now that by construction $d$ and its conjugate all belong to $[G,G]$.
Since $G$ is metabelian, they all commute with each other.
This implies in particular that, in the definition of $\phi$, the order in which the product denoted by $\prod$ is performed has no importance.

In order to establish regularity, it is sufficient to check that $\phi$ is Lipschitz (with respect to some generating set)\footnote{Changing generating sets will not affect the fact that this map is Lipschitz but may change the Lipschitz constant.
} and that it is injective.

\begin{claim}
The map $\phi$ is Lipschitz.
\end{claim}

\begin{proof}
We will take $\{w,s\}$ as generating set for $L$, and some finite arbitrary generating set $S$ of $G$.
We denote by $|\cdot|_S$ the word length defined on $G$ by $S$.

It suffices to check that pairs of elements $\lambda,\tilde\lambda \in L$ at distance $1$ from each other are mapped by $\phi$ at bounded distance.
Two cases has to be considered, one for each generator of $L$:
\begin{itemize}
	\item $\lambda = \big(\oplus_{j \in \zz} \ell_j, i \big)$ and $\tilde\lambda = \big(\oplus_{j \in \zz} \ell_j, i \pm 1 \big)$.
	In this case, it is immediate to see that we have $\phi(\tilde\lambda) = \phi(\lambda) \delta^{\pm 1}$.
	In particular $\phi(\lambda)$ and $\phi(\tilde\lambda)$ are at distance at most $|\delta|_S$.
	\item $\lambda = \big(\oplus_{j \in \zz} \ell_j, i \big)$ and $\tilde\lambda = \big(\oplus_{j \in \zz} \ell_j + \delta_i, i\big)$, where $\delta_i$ denotes the sequence taking the value 0 everywhere except at position $i$ where the value is 1.
	We have $\phi(\lambda) d^{\pm 1} = \big( \prod_{j \in \zz} \delta^j d^{\ell_j'} \delta^{-j} \big) \delta^i d^{\pm1} \delta^{-i} \delta^i = \big( \prod_{j \in \zz\setminus \{i\}} \delta^j d^{\ell_j'} \delta^{-j} \times \delta^i d^{\ell_i'\pm1} \delta^{-i} \big) \delta^i$, where the last equality comes from the fact explained above that the conjugates of $d$ commute with each other.
	In particular, one of $\phi(\lambda) d$ or $\phi(\lambda) d^{-1}$ will equal $\phi(\tilde\lambda)$.
	Thus, $\phi(\lambda)$ and $\phi(\tilde\lambda)$ are at distance at most $|d|_S$.
\end{itemize}
This proves that $\phi$ is Lipschitz.
With the generating sets considered above, the Lipschitz constant obtained is $\max(|\delta|_S,|d|_S)$.
\end{proof}

\begin{claim}
The map $\phi$ is injective.
\end{claim}

\begin{proof}
Let us prove the stronger statement that the map $\psi\circ\phi$ is injective.
Since for every $i,j \in \zz$, $\psi(\delta^j d \delta^{-j}) = (b a^j,1)$ and $\psi(\delta^i) = (0,a^i)$, we have the following expression for $\psi\circ\phi$:
\begin{align*}
  \psi\circ\phi \colon L & \longrightarrow \Gamma \subset k \rtimes k^\times  \\
  \big(\oplus_{j \in \zz} \ell_j, i \big) & \longmapsto \big( \sum_{j \in \zz} (\ell_j' b) a^j, a^i \big)
\end{align*}

Let $\lambda = \big(\oplus_{j \in \zz} \ell_j, i \big)$ and $\tilde\lambda = \big(\oplus_{j \in \zz} \tilde\ell_j, \tilde i \big)$ be two distinct elements of $L$.
We want to prove $\psi\circ\phi(\lambda) \neq \psi\circ\phi(\tilde\lambda)$.
Since $|a|_k \neq 1$, $a$ is not a root of the unity.
Then, it follows from the expression above that we have $\psi\circ\phi(\lambda) \neq \psi\circ\phi(\tilde\lambda)$ whenever $i \neq \tilde i$.
So, we can assume $i=\tilde i$ without any loss of generality.

Let us consider how $\psi\circ\phi(\lambda)$ and $\psi\circ\phi(\tilde{\lambda})$ act on the zero element of the affine line, $0\in k$. 
If $j_M = \displaystyle \max \lbrace j \in \zz \mid \ell_j \neq \tilde{\ell}_j \rbrace$, then 
\begin{align*}
	|\psi\circ\phi(\lambda) \cdot 0 - \psi\circ\phi(\tilde{\lambda}) \cdot 0 |_k 
	&\ge \big| \sum_{j \in \zz} (\ell_j' b) a^j - \sum_{j \in \zz} (\tilde{\ell_j'} b) a^j \big|_k 
	\\&\ge |ba^{j_M}|_k - \big| \sum_{j < j_M} (\ell_j' - \tilde{\ell_j'}) b a^j \big|_k
	\\&\ge \tfrac{1}{2} |ba^{j_M}|_k.
\end{align*}
The last step is follows from the fact that we have $| \sum_{j < j_M} (\ell_j' - \tilde{\ell_j'}) b a^j |_k  \le \frac12 | b a^{j_M}|_k$, a consequence on our assumption that $|a|_k \ge 2$.
This implies the map $\psi\circ\phi$ is injective.
\end{proof}

We have proven that $\phi: L \to G$ is a regular map, which ends the proof of Lemma \ref{lemnous}.
\end{proof}

\subsubsection{Proof of Theorem \ref{teomet}}
\label{sec: proof themeta}

The proof of Theorem \ref{teomet} relies on previous results.
The first one is classical:
\begin{lem}\label{lemti} \emph{(see Tits \cite{Tits} or Breuillard \cite[Lemma 2.2]{Br})}
Let $K$ be a finitely generated field and $\alpha \in K$.
\begin{itemize}
\item If $\alpha^{-1}$ is not an algebraic integer 
(\ie over $\Z$ if $\cha(K) = 0$ or over $\F_p$ if $\cha(K) = p$),
then there exists an embedding $\sigma\colon K \inj k$ into a non-archimedean local field $k$ such that $|\sigma(\alpha)|_k \neq 1$.
\item If $\alpha$ is an algebraic unit which is not a root of unity, then there exists an embedding $\sigma\colon K \inj k$ into an archimedean local field $k$ such that $|\sigma(\alpha)|_k \neq 1$.
\end{itemize}

\end{lem}
The second ingredient goes back to Groves \cite{Gr} but may also be found in Breuillard \cite[Proposition 4.1 in \S{}4.2]{Br}).
We recall that $\A(K)$ denotes the affine group $K \rtimes K^\times$.
\begin{prop}\label{lemgr}
If $G$ is a finitely generated metabelian group of exponential growth, then there is a field $K$ and a map $\rho: G \to \A(K)$ so that the image is also of exponential growth.
\end{prop}

We can now prove Theorem \ref{teomet}.
\begin{proof}[Proof of Theorem \ref{teomet}]
(see Breuillard \cite[\S{}3.1]{Br})
By Proposition \ref{lemgr} there is a field $K$ and a map $\rho: G \to \A(K)$ whose image is finitely generated metabelian but not virtually nilpotent.
Since only the image of $\rho$ matters to us, it can be assumed that $K$ is finitely generated.
Recall $\pi \colon \A(K) \surj K^\times$ the natural projection, which is a group homomorphism.

The group $\Gamma \ce \rho(G)$ being not virtually nilpotent, $Q \ce \pi(\Gamma)$ is not included in the roots of unity of $K$ (otherwise $Q$ would be finite [since it is finitely generated], and $\Gamma$ would be virtually abelian).

We can distinguish two cases:
\begin{itemize}
	\item If $Q$ lies in the subgroup of $K^\times$ consisting of algebraic units\footnote{this corresponds to the case where $\Gamma$ is polycyclic, see \cite[Lemma 3.1 in \S{}3.2]{Br}}, let $\alpha$ be an element of $Q$ which is not a root of unity.
	By Lemma \ref{lemti}, there exists an embedding $\sigma$ of $K$ into an archimedean local field $k$ such that $|\sigma(\alpha)|_k \neq 1$.
	\item If $Q$ is not contained in the subgroup of algebraic units, then there is a $\alpha \in Q$ such that $\alpha^{-1}$ is not an algebraic integer.
	By Lemma \ref{lemti}, there is an embedding $\sigma$ into a non-archimedean local field $k$ with $|\sigma(\alpha)|_k \neq 1$. \end{itemize}
In both cases, we obtain a surjective homomorphism $G \surj \Gamma \subset k \rtimes k^\times$, with $k$ being a local field, such that $\Gamma$ is not abelian and contains an element $\alpha$ satisfying $|\alpha|_k \neq 1$.
Applying Lemma \ref{lemnous}, we obtain a regular map $L \to G$, which ends the proof of Theorem \ref{teomet}.
\end{proof}

\section{Separation profile of the lamplighter group}
\label{sec: separation of lamplighter}
 
The goal of this section is to give an elementary proof of the following proposition:

\begin{prop}\cite[Theorem 4.3]{HMTLie}
\label{propseplamp}
	The lamplighter group has the following separation profile:
	\[\Sep_{\Z_2\wr\Z}(v) \asymp \frac v{\log v}.\]
\end{prop}

\begin{remk}
	More generally, the lamp group~$\Z_2$ can be replaced by any finite group; the same proof applies.
\end{remk}

%
%
%

\begin{proof}
\textbf{Upper bound}
Let us start by proving the upper bound.
Let $F \subset G$ be a finite connected subgraph of $G$ containing $v$ vertices.
Let
\[\pi_\Z : \Z_2 \wr \Z \to \Z\]
be the natural projection map.

Since $F$ is connected, $\pi_\Z(F)$ is connected, thus is an interval $I$.
Moreover, the fact that $F$ is connected implies that we have $f_{|\Z \setminus I} = g_{|\Z \setminus I}$, for every $(f,i),(g,j) \in F$.
This implies:
\[v \le r 2^r,\quad\text{with $r = \# \pi_\Z(F)$.}\]

Let $i_G \in \Z$ be such that
\[\#\{x\in F \mid \pi_\Z(x) < i_G \} \le v/2,\]
and
\[\#\{x\in F \mid \pi_\Z(x) > i_G \} \le v/2.\]

Let $i_G'$ be the biggest integer satisfying $i_G' \le i_G$ and $\# \pi_\Z^{-1}(\{i_G'\}) \cap F \le \frac{4v}{\log v}$.
Similarly, let $i_G''$ be the smallest integer satisfying $i_G'' \ge i_G$ and $\# \pi_\Z^{-1}(\{i_G''\}) \cap F \le \frac{4v}{\log v}$.
Then, there are at most $\frac{8v}{\log v}$ vertices in $C \ce \pi_\Z^{-1}(\{i_G',i_G''\}) \cap F$.
Let us prove that $C$ separates $F$.

By removing $C$, the graph $F$ is disconnected into three parts\footnote{some might be empty or not connected}:

\[F_1 \vcentcolon = \{x\in F \mid \pi_\Z(x) < i'_G \},\]

\[F_2 \vcentcolon = \{x\in F \mid i'_G < \pi_\Z(x) < i''_G \},\]

\[F_3 \vcentcolon = \{x\in F \mid \pi_\Z(x) > i''_G \}.\]

By construction, both $F_1$ and $F_3$ have size at most $v/2$.

\begin{claim}
The connected components of $F_2$ contain at most $v/2$ vertices.
\end{claim}
\begin{proof}
Let us assume as a contradiction that $F_2$ has a connected component $F'$ containing more than $v/2$ vertices.
Let $v' = \# F'$ and $r' = \# \pi_\Z(F')$.
Then, we have
\begin{align*}
v' &> v/2
\\&\ge \#F_2/2
\\&> \frac12 |i''_G - i'_G - 1| \frac{4v}{\log v}\quad\text{by definition of $i'_G$ and $i''_G$}
\\&\ge \frac12 r' \frac{4v}{\log v}
\\&\ge \frac12 r' \frac{4v'}{\log v'}
\\&= r' \frac{2v'}{\log v'}.
\end{align*}
Moreover, since $F'$ is connected, we have as before:
\[v' \le r' 2^{r'},\]
which gives
\[\log v' \le 2r'.\]
Combining, the above inequalities, we obtain:
\begin{align*}
	v' 	&> r' \frac{2v'}{\log v'}
		\\&\ge v',
\end{align*}
which is is a contradiction.
\end{proof}
Since $C$ separates $F$, we obtain the desired upper bound on the separation profile:
\[\Sep_{\Z_2 \wr \Z}(v) \preceq \frac v {\log v}.\]

\textbf{Lower bound}
Let us prove now the lower bound.
For each positive integer $n$, let us consider the subgraph $T_n$ induced by the set of vertices given by a lamplighter position in the interval $[-n,n]$ and the lamps of this interval being arbitrarily on or off.
In other word, $T_n$ is the set $\Z_2^{[-n,n]} \times [-n,n]$. It has $(2n+1)2^{2n+1}$ vertices.
Let us compute the separation of this graph.

In the spirit of the proof of~\cite[Theorem 3.5]{BST}, let~$W$ be a separating set for $T_n$ and let $x = (f,i)$ and $y = (g,j)$ be two elements of~$T_n$.
We consider the following path from $x$ to $y$:
\begin{enumerate}
	\item\label{enum: C1} from~$(f,i)$ to~$(f,-n)$:~$i+n$ multiplications by~$w^{-1}$.
	\item\label{enum: C2} from~$(f,-n)$ to~$(g,n)$:~$2n$ multiplications by~$w$, plus one multiplication by~$s$ each time we have~$f_k\ne g_k$.
	\item\label{enum: C3} from~$(g,n)$ to~$(g,j)$:~$n-j$ multiplications by $w$.
\end{enumerate}

Let us call this path~$P(x,y)$. Clearly,~$P$ contains at most~$8n$ elements.
If $x$ and $y$ are picked independently and uniformly at random in $T_n$, the path $P(x,y)$ passes through the separating set~$W$ with probability at least~$1/2$.

This means that, among the $|T_n|^2$ path we have defined, at least $\frac12 |T_n|^2$ of them intersect $W$.
Following the approach of \cite[Proposition 1]{GladkovaShum}, we will use the following claim:

\begin{claim}
	Each vertex of $T_n$ lies in at most $\frac{3|T_n|^2}{2^{2n+1}}$ paths of the family $P(x,y)$.
\end{claim}
\begin{proof}
	Let $z=(h,k) \in T_n$.
	Let $x=(f,i)$ and $y=(g,j)$ be such that $z$ lies in the path $P(x,y)$ joining $x$ and $y$.
	Let $C_1$, $C_2$ and $C_3$ be the set of vertices encountered when following the subpaths \ref{enum: C1}, \ref{enum: C2}, and \ref{enum: C3} respectively.
	\begin{itemize}
		\item If $z$ lies in $C_1$, this implies that we have $f=h$.
		\item If $z$ lies in $C_2$, this implies that there is some $l\in [-n,n]$ such that we have $g_{|[-n,l]} = h_{|[-n,l]}$ and $f_{|[l+1,n]} = h_{|[l+1,n]}$.
		\item If $z$ lies in $C_3$, this implies that we have $g=h$.
	\end{itemize}		
		In each of the three situations above, there are at most $\frac{|T_n|^2}{2^{2n+1}}$ pairs $(x,y) \in T_m \times T_n$ satisfying the condition.
		This implies that $z$ cannot lie in more that $\frac{3|T_n|^2}{2^{2n+1}}$ paths of the family $P(x,y)$.
\end{proof}

Then, $W$ has to contain at least $\frac{|T_n|^2}{2 \cdot 3|T_n|^2/2^{2n+1}}$ vertices.
Note that we have 
\begin{align*}
	\frac{|T_n|^2}{2 \cdot 3|T_n|^2/2^{2n+1}} &= \frac16 \cdot 2^{2n+1}
	\\&= 	\frac16\frac{T_n}{2n+1}
	\\&\ge 	\frac16\frac{T_n}{\log |T_n|}
\end{align*}
This shows that we have~$\Cut T_n \ge \frac16\frac{|T_n|}{\log |T_n|}$, which implies the desired lower bound on the separation profile:
\[\Sep_{\Z_2 \wr \Z}(v) \succeq \frac v {\log v}.\qedhere\]
\end{proof}

\bibliography{./bibliothaeque}{}
\bibliographystyle{abbrv}
\end{document}